\documentclass{daj}

\usepackage{amsmath}
\usepackage{amsthm, amssymb, amsfonts}

\def\eps{\varepsilon}
\def\p{\partial}	
\def\HH{\mathbb{H}}
\def\E{\mathbb{E}}
\def\C{\mathbb{C}}
\def\R{\mathbb{R}}
\def\Z{\mathbb{Z}}

\def\PP{\mathbb{P}}
\def\l{\lambda}

\def\s{\sigma}
\def\t{\theta}

\def\g{\gamma}
\def\z{\zeta}	
\def\zbar{\bar{z}}

\def\Var{\mathrm{Var}}

\def\vp{\varphi}
\def\im{\mathrm{Im}}
\def\re{\mathrm{Re}}
\renewcommand{\Re}{\re}
\renewcommand{\Im}{\im}
\renewcommand{\leq}{\leqslant}
\renewcommand{\geq}{\geqslant}

\renewcommand{\to}{\rightarrow}
\renewcommand{\P}{\mathbb{P}}
\renewcommand{\arg}{\mathrm{arg}}

	\newtheorem{theorem}{Theorem}
	\newtheorem{lemma}[theorem]{Lemma}
	
	\newtheorem*{claim*}{Claim}
	\newtheorem{corollary}[theorem]{Corollary}
	
	\newtheorem{fact}[theorem]{Fact}
	\newtheorem{obs}[theorem]{Observation}

	\theoremstyle{definition}

	\newtheorem{conjecture}[theorem]{Conjecture}

	\theoremstyle{remark}

\dajAUTHORdetails{%
  title = {Anti-concentration of random variables from zero-free regions}, %% please capitalize all significant words
  author = {Marcus Michelen and Julian Sahasrabudhe},
  plaintextauthor = {Marcus Michelen, Julian Sahasrabudhe},}
\dajEDITORdetails{%
   year={2022},
   number={13},
   received={2 September 2021},   % received date: example: 7 January 2017
   published={7 October 2022},  % published date
   doi={10.19086/da.38590},       % XXX = number of paper, e.g. da006 for paper#6
}   %%% END \dajEDITORdetails

\begin{document}

\begin{frontmatter}[classification=text]
%% EDITOR: this will force the keywords to appear right after the Abstract.
%%   If the abstract is too long and would force the keywords off the
%%   front page, please comment out % [classification=text] above
%%   This way the keywords will be floated on the bottom of the first page
%%   even though the Abstract spills over to the next page.

%%% AUTHOR: Title goes here.  This line is optional.  You must use it
%%   if title has footnote attached or requires nontrivial typesetting,
%%   e.g., inclusion of linebreaks to force nice layout.
%\title{A Characterization of Polynomials Whose High Powers Have Non-negative Coefficients} %% please capitalize all significant words

%%% AUTHOR:
%%% List all authors. If you wish, place grant acknowledgements in \thanks.
%%% In brackets include a short tag for each author.
\author[marcus]{Marcus Michelen\thanks{Supported in part by NSF grant DMS-2137623.}}
\author[julian]{Julian Sahasrabudhe}
%\author[laci]{L\'aszl\'o Lov\'asz\thanks{Supported by...}}
%\author[andy]{Andrew Chi-Chih Yao\thanks{Supported by...}}

%%% AUTHOR: Abstract goes here
\begin{abstract}
This paper provides a connection between the concentration of a random variable and the distribution of the roots of its probability generating function.
Let $X$ be a random variable taking values in $\{0,\ldots,n\}$ with $\P(X = 0)\P(X = n) > 0$ and with probability generating function $f_X$. We show that if all of the zeros $\z$ of $f_X$ satisfy $|\arg(\z)| \geq \delta$ and $R^{-1} \leq |\z| \leq R$ then
\[ \Var(X) \geq c  R^{-2\pi/\delta}n,
\] where $c > 0$ is a absolute constant.  We show that this result is sharp, up to the factor $2$ in the exponent of $R$. 
As a consequence, we are able to deduce a Littlewood--Offord type theorem for random variables that are not necessarily sums of i.i.d.\ random variables.
\end{abstract}
\end{frontmatter}

%%% AUTHOR: body of paper starts here

\section{Introduction}

While there are many tools in probability theory for showing that a random variable is concentrated, there are few for proving \emph{anti-concentration} in a general setting.  One family of results in this direction is Littlewood--Offord theory \cite{erdos-LwO,halasz,LittlewoodOfford,stanley,rogozin}, also known as small ball probability, which is a set of tools for obtaining upper bounds on the probability that a random sum is in a ``small'' set.  This line of work has led to \emph{inverse Littlewood--Offord theorems} \cite{nguyen-vu,tao-vu,rudelson2008littlewood} which often present a useful dichotomy: either a certain random variable exhibits this anti-concentration or a special structure is present.  This approach has been extended to low degree polynomials \cite{meka2016anti,kwan2019algebraic} although sharp results remain elusive in these non-linear cases. 

Another route towards anti-concentration is by a coupling approach: when the variable of interest is a function of a random environment, one can often couple two instances of the environment so that one instance of the variable is larger than the other.  This approach was taken by Wehr and Aizenman \cite{wehr-aizenman} to yield lower bounds on certain variances in the context of the Ising model (and other related models) and is also taken up in other ad-hoc approaches to proving lower bounds on fluctuations~\cite{bollobas-janson,gong-houdre-lember,houdre-ma,houdre-matzinger,janson,lember-matzinger,rhee} which culminated in a recent unifying work of Chatterjee \cite{chatterjee}. 

For lower bounds specifically on the variance, there are also a few other tools available; the Cram\'er--Rao inequality  \cite{cramer,rao} and a related approach by Cacoullos \cite{cacoullos} provide variance lower bounds for functions of i.i.d.\ random variables in terms of Fisher information. 

While these approaches are powerful, they all depend deeply on interpreting the random variable of interest as a function of a random environment, typically a family of i.i.d.\ variables, and thus do not apply to variables without such an interpretation. 

In this paper, we prove anti-concentration estimates for a random variable based solely on the location of the roots of its probability generating function. 
For a random variable $X \in \{0,\ldots,n\}$, we let 
$$ f_X(z) := \sum_{k} \PP(X = k)z^k $$
be its probability generating function. We shall write $p_i = \PP( X = i)$, when $X$ is clear from context.

Our original motivation derives from a conjecture of Pemantle (see \cite{clt1})
and a related conjecture of Ghosh, Liggett and Pemantle \cite{GLP} on random variables with real stable probability generating functions. Pemantle conjectured that random variables 
$X \in \{0,\ldots,n\}$ for which $f_X$ has no roots in a sector $\{ z : |\arg(z) | < \delta \}$ are approximately normal, provided $\s(X) \gg 1$. 
We formulated the following natural conjecture which implies, when applied with ideas from \cite{LPRS,clt1}, an important subcase of Pemantle's conjecture, when the roots are bounded away from $0$ and $\infty$.

\begin{conjecture}\label{conj:lin-var}
	Let $X \in \{0,\ldots,n\}$ be a random variable with $p_0 p_n > 0$ and probability generating function $f_X$
	and let $\delta > 0 $, $R \geq 1$. If the zeros $\z$ of $f_X$ satisfy $ |\arg(\z)| \geq \delta$ and $R^{-1} \leq |\z| \leq R$ then 
	\[ \Var(X) = \Omega_{R,\delta}( n ). \]
\end{conjecture}

While we ultimately resolved the conjecture of Pemantle by different means (see \cite{clt2}), Conjecture~\ref{conj:lin-var} remains of independent interest
and motivates our work here.
What is perhaps surprising about this conjecture is that it says that random variables $X$ of this type have variance that is 
essentially as large as possible, as it is not hard to see that $\Var(X) = O_{R,\delta(n)}$. Indeed, under the assumptions of Conjecture \ref{conj:lin-var}, 
we have
$$\Var(X) = \frac{d^2}{d t^2}\log f_X(e^{\theta})\, \bigg|_{t = 0} = \sum_{j = 1}^n \frac{d^2}{d t^2}\log |e^t - \zeta_j |\, \bigg|_{t = 0} = O_{R,\delta}(n)\,, $$
where $\z_1,\ldots,\z_n$ are the roots of $f$. Thus, Conjecture~\ref{conj:lin-var} implies that the variance of such random variables is \emph{determined} up to constant factors.

In this paper we not only prove Conjecture~\ref{conj:lin-var}, but supply a near-optimal constant. In other words, we give a near-optimal lower bound for the variance of $X$ based on the smallest argument of a root and the smallest annulus that contains the zeros of $f_X$. 

\begin{theorem}\label{thm:Var-lower-bound} Let $X \in \{0,\ldots,n\}$ be a random variable with $p_0 p_n > 0$ and probability generating function $f_X$
	and let $\delta > 0 $, $R \geq 1$. If the zeros $\z$ of $f_X$ satisfy $ |\arg(\z)| \geq \delta$ and $R^{-1} \leq |\z| \leq R$ then
	\[ \Var(X) \geq  c R^{-2\pi/\delta} n ,\]
	where $c>0$ is an absolute constant.  
\end{theorem}
We will in fact prove a slightly stronger lower-bound, but we postpone this more technical statement to Section~\ref{sec:proof}.

\subsection{Some corollaries of Theorem~\ref{thm:Var-lower-bound}}

In \cite{clt1,clt2} we studied the relationship between zero-free regions of probability generating functions and central limit theorems, culminating in two sharp central limit theorems\footnote{See Lebowitz, Pittel, Ruelle and Speer's work \cite{LPRS} for earlier results on the relationship between zero-free regions and central limit theorems.}.  In \cite{clt2}, we resolved Pemantle's conjecture by showing that if the generating function of a random variable $X$ has no roots with argument less than $\delta$, then $X$ is approximately Gaussian provided $\Var(X) \gg \delta^{-2}$.  Combining this \cite[Theorem 1.4]{clt2} with Theorem \ref{thm:Var-lower-bound} %actually yields a more comprehensive result than Corollary~\ref{cor:LwO} about the distribution of a random variable with constrained roots; it 
allows us to prove a quantitative central limit theorem for $X$.

\begin{corollary}\label{cor:Gauss-flux}
	Let $X \in \{0,\ldots,n\}$ be a random variable with $p_0p_n> 0$, mean $\mu$, variance $\s^2$ and probability generating 
	function $f_X$. Also let $\delta >0$, $R \geq 1$ and set $X^{\ast} := (X-\mu)\s^{-1}$.  If the zeros $\z$ of $f_X$ satisfy $ |\arg(\z)| \geq \delta$ and $R^{-1} \leq |\z| \leq R$ then
	$$\sup_{t \in \R} \left|\P\left( X^{\ast} \leq t \right) - \P(Z \leq t) \right| \leq c  \delta^{-1} R^{\pi/\delta} \cdot n^{-1/2}, $$
	where $Z$ is a standard Gaussian random variable and $c > 0$ is an absolute constant.
\end{corollary}

Thus, while Theorem \ref{thm:Var-lower-bound} shows a lower bound on the variance of $X$, Corollary~\ref{cor:Gauss-flux} allows us to deduce that $X$ also has \emph{fluctuations} on the order of $\sqrt{\Var(X)} = \Theta_{R,\delta}(\sqrt{n})$, provided $n$ is large relative to $R^{2\pi/\delta}$.

%As a consequence of Theorem~\ref{thm:Var-lower-bound}, together with a theorem from our previous paper \cite{clt2}, 
From Corollary \ref{cor:Gauss-flux}, we obtain a Littlewood--Offord type theorem for general random variables. To understand this result in the context of Littlewood--Offord theory, we recall the classical result of Erd\H{o}s~\cite{erdos-LwO} which says that if $X$ is of the form $X = \sum_{i=1}^n \eps_iv_i$, 
where $\eps_1,\ldots,\eps_n \in \{-1,1\}$ are i.i.d. uniform and $v_1,\ldots,v_n \in \R$ are non-zero then 
\[ \max_{y}\, \PP( X = y ) = O(n^{-1/2}). \]
The following corollary of Theorem~\ref{thm:Var-lower-bound} says that a similar result holds even if $X$ is \emph{not} a sum of independent random variables: one needs only some control on the 
roots of the probability generating function of $X$.

\begin{corollary}\label{cor:LwO}
	Let $X \in \{0,\ldots,n\}$ be a random variable with $p_0 p_n > 0$ and probability generating 
	function $f_X$ and let $\delta >0$, $R\geq 1$. If the zeros $\z$ of $f_X$ satisfy $ |\arg(\z)| \geq \delta$ and $R^{-1} \leq |\z| \leq R$, then
	$$\max_y\, \P(X = y) \leq c \delta^{-1} R^{\pi/\delta} n^{-1/2} $$ 
	where $c > 0$ is an absolute constant.
\end{corollary}

Again this is best possible, up to a factor of $2$ in the exponent of $R$. The reader may have anticipated a upper bound of the form $c(\Var(X))^{-1/2}$ in Corollary~\ref{cor:LwO}, in the extreme case, which is a factor of $\delta^{-1}$ off of what we have. However this additional factor \emph{is} necessary; a fact which can, for example, be seen in the example of sharpness in Section~\ref{subsec:discuss}.

\subsection{Zeros of generating functions} \label{subsec:zeros}

Perhaps surprisingly, many families of random variables, which are otherwise elusive, are known to have generating functions with zero-free regions.  A classical instance is provided in the highly influential pair of 1952 works of Lee and Yang \cite{lee-yang,yang-lee}, which showed that the roots of the partition function in the ferromagnetic Ising model lie on the unit circle. That is, in our terminology, the probability generating function for ``up spins'' in the Ising model has all of its roots on the unit circle.  Another classical example is provided by Heilmann and Leib \cite{heilmann-leib70}, who showed that a similar special property is enjoyed by the random variable $X = |M|$ where $M$ is a uniformly chosen matching in a graph $G$. In this case they showed that the  roots of the probability generating function are known to be \emph{real}.

These results on zeros have significant implications for the study of these random systems. Lee and Yang connected the theory of zero-freeness to the non-existence of a phase transition: the ferromagnetic Ising model cannot have a phase transition as the external field $h$ varies, except at $h = 0$. The work of Heilmann and Leib yields similar results for the study of ``monomer-dimer'' systems and, in fact, here tells us that there is no phase transition at all.  While there are other techniques for ruling out phase transitions for the monomer-dimer model on amenable graphs such as $\Z^d$ \cite{van-den-berg} (see \cite{lebowitz-martin,preston} for analogous results for the Ising model ), the zero-free approach remains the most robust route to proving there is no phase transition for these models on arbitrary graphs. 

From a more combinatorial perspective, Godsil \cite{godsil81}, in a classic work,
used the work of Heilmann and Leib to obtain a central limit theorem for the size of a random matching in a $d$-regular graph for fixed $d$. More recently, this
was taken much further by Kahn \cite{kahn00} who, essentially relying on the work of Heilmann and Leib, gave a nearly complete understanding of this phenomena for matching in graphs.

%Recent work of Peters and Regts \cite{peters-regts-Ising} extended these classical results of Lee and Yang by showing that for weak enough ferromagnetic interactions, the Ising model on maximum degree $d$ graphs is zero free in a sector $\{ z \in \C : |\arg(z)| < \delta_d \}$, thus satisfying the assumptions of Theorem \ref{thm:Var-lower-bound} and its Corollaries. To be more concrete, let $X$ be the number of up-spins in the Ising model on a bounded degree graph $G$, with $n$ vertices and maximum degree $d$. Under the presence of an external field and small ferromagnetic interactions, our Theorem \ref{thm:Var-lower-bound} implies $\Var(X) = \Omega(n)$ and Corollary \ref{cor:LwO} implies $\max_k \PP(X = k) = O(n^{-1/2})$, where the implicit constants depend only the degree $d$, the external field and the inverse temperature.

Recently, it has shown to be fruitful to consider multivariate (or in the case of random variables, \emph{multi-dimensional}) versions of these polynomials.
In the case of both the Ising and monomer-dimer models, \emph{multivariate} zero-free regions are now known and a general theory has developed around these results.
One particular success has been had with \emph{stable polynomials}, which emerged in probability theory with the uprising work of Borcea, Br\"anden and Liggett \cite{BBL}, who connected polynomials to a natural notion of negative dependence set out by Pemantle in his influential work on the subject \cite{pemantle2000}.

We call a random variable with a real stable generating function \emph{strong Rayleigh} and since the work of Borcea, Br\"anden and Liggett \cite{BBL} many random variables have been shown to be strong Rayleigh, such as the edges of a uniform spanning tree, independent Bernoulli random variables with conditioned sum, and more \cite{pemantle-survey}.  Our original motivation was to solve a question raised by Ghosh Liggett and Pemantle \cite{GLP} on the limiting shape of strong Rayleigh distributions. While we now know these distributions approximate multivariate gaussians \cite{clt2}, Theorem~\ref{thm:Var-lower-bound} allows us to understand the scale of this normal shape, in all directions.  More specifically, if the generating function $p(z_1,\ldots,z_d)$ of a random variable $(X_1,\ldots,X_d)$ is stable, then the random variable $m_1 X_1 + \cdots + m_d X_d$ for non-negative integers $m_j$ is zero-free in the sector $\{|\arg(z)| < \pi / \max m_j  \}$.  Theorem \ref{thm:Var-lower-bound} then shows that no positive linear combination of $(X_1,\ldots,X_d)$ can be degenerate, assuming some control over the maximum and minimum root.

Another light in which to view our results comes from the analytic theory of characteristic functions as studied by Yu. V. Linnik, Ostrovskii and others. 
We refer the reader to \cite{LinnikConj} and the book of Linnik and Ostrovskii \cite{linnikBook} and the references therein for more detail on this fascinating line of research.

%Beyond the models outlined above, zero-free regions have been demonstrated for various other models from probability, combinatorics and statistical mechanics.  Rather than attempt a comprehensive survey we point towards a few models for which zero-freeness plays a central role: a zero-free region for the hard-core model was first shown by Shearer {\RED cite} and recently expanded by Peters and Regts (see also Bencs-Csikvari); further, the zeros of the hard-core model were shown by Scott and Sokal to have deep connections to the LLL.  A zero-free region was 
%Interestingly, many other families of random variables, which are in other ways illusive, are known to admit zero-free regions. Perhaps most classically is the highly influential work of Lee and Yang \cite{lee-yang,yang-lee}, who showed that the roots of the partition function (a probability generating function for ``up spins'') in the ferromagnetic Ising model lie on the unit circle. More recently, Peters and Regts \cite{peters-regts-Ising} have showed that such partition functions (in both ferromagnetic and anti-ferromagnetic) cases avoid a sector in the unit circle
%of width depending only on the degree of the underlying graph.  We refer the interested reader to \cite{heilmann-leib72,LPRS,peters-regts-hardcore} and the references therein, for further interesting examples of distributions with zero-free regions.
%
%
\subsection{Sharpness of results}\label{subsec:discuss}

The sharpness of Theorem \ref{thm:Var-lower-bound} (and Corollaries~\ref{cor:LwO} and~\ref{cor:Gauss-flux}), up to a factor of $2$ in the exponent of $R$, is supplied by
a natural class of random variables, which we describe here. 

Given $R \geq 1$ and $\delta = \pi/k$ for some integer $k \geq 3$. We choose $p = (1 + R^k)^{-1}$ and let $X_1,\ldots,X_{ n/k }$ be independent, identically distributed Bernoulli random variables where $p = \PP(X_j = 1)$ and thus $\Var(X_j) = p(1-p)$.  Now define $X$ to be the sum 
$$X = k \sum_{j = 1}^{ n/k }X_i.$$ 
One can then see that $f_X(z) = (p z^k + (1 - p))^{ n/k}$ and thus all roots have modulus $R$ and argument $\geq \pi/k$. It is not hard to additionally show that 
\[ \Var(X) = \Theta\left( \delta^{-1} R^{-\pi/\delta}n \right),  \]
demonstrating that Theorem~\ref{thm:Var-lower-bound} is sharp up to the factor of $2$ in the exponent\footnote{Interestingly, the extra factor of $\delta^{-1}$ appears in our more detailed technical statement, Theorem~\ref{thm:Var-lower-bound-sharper}, while the exponent remains off by a factor of $2$.}  of $R$. Likewise, one can show that this example satisfies
\[ \max_{y}\, \PP( X  = y ) = \Theta\left( \delta^{-1/2}R^{\pi/2\delta} n^{-1/2} \right), \]
provided $(n/k) R^{-k} \gg 1$. Thus it remains an interesting open problem to close the gap between this example and Theorem~\ref{thm:Var-lower-bound} and Corollary~\ref{cor:LwO}.

\subsection{Outline of proof}

The proof of Theorem~\ref{thm:Var-lower-bound} is broken into three principal steps.   The first step draws on the results of our paper \cite{clt2} and is carried out in Section~\ref{sec:Var-and-phi} where we relate $\Var(X)$ to the value of
the function
$$\vp_{\g}(z) = \log| f_X(z )| - \log |f_X(ze^{i\gamma})|$$ at $z=1$.  The function  $\vp_{\g}$ has several nice properties that will be crucial for us:
we will see that if $\gamma \approx c\delta$, the function $\vp_\gamma$ is both positive and harmonic in some sector of the positive real axis and, in addition, the Taylor expansion of $\varphi_\gamma(1)$ in variable $\gamma$ has leading term $\gamma^2 \Var(X)/2$, thus providing the link with the variance. Bounding the variance in terms of $\vp_\g(1)$ then amounts to showing that higher terms in this Taylor expansion may be disregarded. However, removing these ``higher'' terms is no small matter and it is at this point that we make essential use of the tools built up in our previous paper \cite{clt2} which allow us to lower bound $\Var(X)$ in terms of $\vp_\gamma(1)$ (Lemma~\ref{lem:Var-and-phi}) by controlling the higher cumulants in terms of $\Var(X)$. Indeed, heavy use of the fact that $f_X$ has non-negative coefficients is used in this step.  After Lemma \ref{lem:Var-and-phi} is in place, our path diverges from the ideas and results in \cite{clt2}.  

The second step in the proof of Theorem~\ref{thm:Var-lower-bound} is carried out in Section~\ref{sec:connection-to-mellin} where we obtain a lower bound for $\vp_{\g}(1)$ (Lemma \ref{lem:phi-at-least-H_r}) in terms of the value of a certain ``truncated Mellin transform'' $H_{M,\tau}(s)$. While at this point this is a somewhat mysterious step, we will see later that this  Mellin transform has some useful properties in our context, that will allow us to get a good handle on it.

To relate $\vp_{\g}(1)$ to this truncated Mellin transform we (after some preparations) use the fact that $\log|f_X(z)|$ is harmonic in the sector $\arg(\zeta) \in (0,\delta)$ to write $\vp_\g(r e^{i(\delta-\gamma)})$ as an integral around the boundary of the sector against some Poisson kernel $P(t)$. We are then able to truncate this integral and do a direct comparison to the similar-looking Mellin transform. Again we make use of the non-negativity hypothesis in this step.

The final step, which is presented in Sections~\ref{sec:mellin-comp} and~\ref{sec:truncation}, is to control the value of this truncated Mellin transform $H_{M,\tau}(s)$. The key ingredient here is that in our situation of constrained roots, we have very good control over this object for small $s$. Indeed, when we take $M \to \infty$ and $s \to 0$, we have that $H_{M,\tau}$ approaches $n\tau^2$, thereby providing the factor of $n$ in Theorem \ref{thm:Var-lower-bound}.  We compute this limit by first calculating the (non truncated) Mellin transform \emph{exactly}, in the $s \rightarrow 0$ limit. We then show that we can (carefully) truncate the integral without too much loss.

The three lower bounds, $\Var(X)$ in terms of $\varphi_\gamma(1)$, $\varphi_\g(1)$ in terms of $H_{M,\tau}(s)$, and $H_{M,\tau}(s)$ in terms of $n \tau^2$, are then assembled in Section \ref{sec:proof} to prove Theorem \ref{thm:Var-lower-bound}.  

The main contribution of this paper, broadly speaking, is to understand how the constraints on the roots of $f_X$ and non-negativity of the coefficients of $f_X$ interact. Interestingly, using \emph{only} information on the zeros, is not enough, only by using the full strength of this interaction are we able to deduce our results. Indeed, one can interpret the results of this paper as developing a tool kit for understanding how the non-negativity of $f_X$ interacts with information on the location of the roots.  

\section{Basic definitions and properties} \label{sec:definitions}
In this section we introduce some of the central objects in this paper and state their basic properties. 
We refer the reader to our paper \cite{clt2} for a more careful treatment of some of the basic results mentioned in this section. 

For $z\in \C \setminus \{0\}$, we write $z = re^{i\t}$, where $r>0$ and $\theta \in [-\pi,\pi]$ and then define $\arg(z) = \theta$.
For $-\pi \leq \alpha \leq \beta \leq \pi$, we define the \emph{sector} 
\[S (\alpha,\beta) = \{ z \in C : \arg(z) \in [\alpha,\beta] \};\] 
and define $S(\delta) = S(-\delta,\delta)$. We use the notation $f(x) = O(g(x))$ to denote $|f(x)| \leq Cg(x)$, for a positive constant $C$
and we use the notation $o_{x \rightarrow 0}(1)$ to denote a quantity that tends to zero as $x \rightarrow 0$.

For a random variable $X \in  \{0,\ldots,n\}$ we let $f_X$ be its probability generating function 
and define the \emph{logarithmic potential of} $X$ to be 
\[ u_X(z) = \log|f_X(z)|. 
\] One of the reasons for the use of the logarithmic potential is immediately apparent: if $f_X$ has no zeros in an open set $\Omega$, then $u_X$ is a
harmonic function on $\Omega,$ allowing us to appeal to tools available for working with harmonic functions. 

We now note and define a few basic properties of $u_X$. If $u$ is a function defined on $S(\alpha,\beta)$,
we say that $u$ is \emph{symmetric on} $S(\alpha,\beta)$ if $u(z) = u(\bar{z})$, whenever $z,\bar{z} \in S(\alpha,\beta)$.
Of course, if $u_X$ is the logarithmic potential of a random variable, then $u_X$ is symmetric due to the fact that $f_X$ has real coefficients: indeed, write
\[ u_X(z) = \log|f_X(z)| = \log|\overline{f_X(z)}| = \log|f_X(\bar{z})| = u_X(\bar{z}). \]  
We now introduce a key notion that captures the property that $f_X$ has positive coefficients in terms of the logarithmic potential. We say that a function $u$ defined on a sector $S(\alpha,\beta)$ is \emph{weakly positive} if 
\begin{equation} \label{eq:weak-positivity} u(|z|) \geq u(z), 
\end{equation} for all $z \in S(\alpha,\beta)$.
We shall make essential use of the fact that the logarithmic potential of a random variable is weakly positive; indeed, since $f_X$ has non-negative coefficients 
we have that 
\[ |f_X(|z|)| \geq |f_X(z)|, 
\] for all $z \in \C$. Weak positivity of $u_X$ follows by taking logarithms of both sides. The notion of weak-positivity has been studied in several papers before  \cite{BE,deAngelis2,deAngelis3,deAngelis1,eremenko-fryntov,MS-strongPos}; in particular, Bergweiler and Eremenko \cite{BE} showed that weak positivity characterizes the logarithmic potentials of polynomials with non-negative coefficients up to limits.  

We now introduce\footnote{This is the special case $b = 0$ of our notion of $b$-\emph{decreasing} from our previous paper \cite{clt2}.} an essential definition for the work in this paper. We say that a function $u$, defined on the sector $S(\delta)$, is 
\emph{rotationally decreasing} if $u(\rho e^{i\t})$ is a decreasing function of $\t \in [0,\delta]$, for all $\rho >0$.
As we shall see in Section~\ref{sec:Var-and-phi}, the properties of being weakly positive and harmonic in $S(\delta)$ combine nicely to give us this
enriched positivity property.

We shall also draw upon a simple expansion of $u_X$ in terms of the roots of $f_X$. In particular, we have  
\begin{equation} \label{eq:exp-u-with-roots} 
	u_X(z) = \sum_{|\zeta| < 1} \log\left|1 - \frac{\zeta}{z} \right| + \sum_{|\z| \geq 1 }\log\left|1 - \frac{z}{\zeta}\right| + N_X\log|z| + c_X  \,, 
\end{equation} where $N_X := |\{ \z : |\z| < 1 \}|$, the sums are over the roots $\z$ of $f_X$ and $c_X \in \R$ is defined so that $u_X(1) = 0$.
This last property is due to the fact that $f_X(1) = 1$.

Since we will work in the case of $u_X$ harmonic in $S(\delta)$, we will often use the theory of Poisson integration, in which we write a value $u_X(z)$ in terms of an integral along the boundary of $S(\delta)$.  The sector $S(\delta)$ is unbounded and so we will need some basic control over the asymptotic growth of $u_X(z)$ as $z \to \infty$ and $z \to 0$, both of which will be readily available.
We say that a function $u$ on a sector $S(\alpha,\beta)$
has \emph{logarithmic growth} if 
\[ u(z) = O(\log |z|) \text{ as }z \to \infty, \textit{ and  } \, u(z) = O(\log |z|^{-1} ) \text{ as } z \to 0 ,\]
for $z \in S(\alpha,\beta)$. Notice that since $f_X$ is a polynomial of degree at most $n$, $u_X$ has logarithmic growth.

We now introduce an important companion to $u_X$ in this paper, the function $\vp_{\g} = \vp_{\g,u}$. For $\g \in (0,\delta)$ and a function $u$ on $S(\delta)$,
define 
\begin{equation} \label{eq:def-of-phi} \vp_{\g}(z) := u(z) - u(e^{i\g}z). 
\end{equation} The importance of $\vp_{\g}$ comes jointly from the fact that it is  both positive and harmonic in a sector
and its leading term in the series expansion of $\vp_\gamma(1)$ is $\gamma^2 \Var(X) / 2$. This second observation will be noted in Section~\ref{sec:Var-and-phi} and
we record the first observation here.

\begin{obs} \label{obs:basic-vp-facts} For $\delta >0$ and $\g \in (0,\delta)$, let $u$ be a symmetric function on $S(-\delta,\delta)$ 
	and put $\vp_{\g}(z) = \vp_{\g,u}(z)$.  
	\begin{enumerate}
		\item If $u$ is a harmonic function on $S(\delta)$ then $\vp_{\g}(z)$ is harmonic on $S(-\delta,\delta-\g)$;
		\item If $u$ is rotationally-decreasing in the sector $S(\delta)$ then $\vp_{\g}$ is positive in $S(-\g/2,\delta-\g)$.
	\end{enumerate}
\end{obs}
\noindent This observation is not hard to check, but can also be found in \cite[Section 3]{clt2}.

\section{Relating $\vp_\g$ to the variance of $X$} \label{sec:Var-and-phi}

In this section we prove that the variance of $X$ can be lower-bounded by $\vp_{\g}(1)$ (defined at \eqref{eq:def-of-phi})
for $\g \in (0,\delta/2^{6})$. 

\begin{lemma} \label{lem:Var-and-phi} For $\delta \in (0,\pi)$ and $\g \in (0,\delta/2^{6})$, let $X \in \{0,\ldots,n\}$ be a random variable with $\Var(X) > 0$, for which $f_X(z)$ has no zeros in $S(\delta)$. Then
	\[ \Var(X) \geq c \g^{-2} \vp_{\g}(1),
	\] where $c>0$ is an absolute constant and $\vp_{\g} = \vp_{\g,u}$.
\end{lemma} 

To prove this, we rely heavily on tools developed by the authors in \cite{clt2}. Our first step is to use the following lemma which tells us that
$u_X$ is rotationally decreasing in a sector.

\begin{lemma} \label{lem:decreasing} For $\delta >0$, let $u$ be a weakly positive, symmetric and harmonic function on the sector $S(\delta)$, which has logarithmic growth.
	Then $u$ is rotationally decreasing in $S(\delta/2)$. \end{lemma}

\begin{proof}  This lemma follows from Lemma 4.1 in \cite{clt2}, by applying the lemma for all $b \rightarrow 0$ and taking $r$ sufficiently large so that the condition is satisfied.\end{proof}

\vspace{2mm}

We now turn to note a useful series expansion of $u(e^{w})$. We refer the reader to  \cite[Lemma 3.1]{clt2} for a more detailed proof of the facts stated here. 
If $u$ is symmetric and harmonic in $S(\delta)$ we may express $u(e^{w})$ as a series for all $w$ in a neighborhood of $0\in \C$. Indeed, if $u(1) = 0$ we may write
\begin{equation} \label{eq:U-series} u(e^{w}) = \sum_{j\geq 1} a_j\Re\left( w^j \right), 
\end{equation} for all $w \in B(0,\delta)$,
where $(a_j)_{j\geq 1}$ is a sequence of real numbers. The sequence $(a_j)_{j\geq 1}$ is very closely related to the \emph{cumulant sequence} of the random variable $X$, and in particular 
\begin{equation} \label{eq:var-is-2nd-cumulant} \E X = a_1, \quad \Var(X) = a_2/2. 
\end{equation}  We call this sequence $(a_j)_{j\geq 1}$ the \emph{normalized cumulant sequence of} $X$ and, more generally, of a symmetric harmonic function on $S(\delta)$.

Using the definition of $\vp_\g = \vp_{\g,u}$ at \eqref{eq:def-of-phi} along with \eqref{eq:U-series} we obtain an expansion for $\vp_\g$
\begin{equation} \label{eq:phi-series} \vp_{\g}(e^{w}) = \sum_{j\geq 2 } a_j\Re\left(w^j - (i\g + w)^j\right),
\end{equation} for sufficiently small $w$, when $\g \in (0,\delta/2)$.
We will then apply the following lemma which tells us that if a function is decreasing, weakly positive and symmetric we can obtain
very tight control over the tail of the normalized cumulant sequence. The following lemma is the special case of $b=0$ of Lemma 6.1 in our paper \cite{clt2}.

\begin{lemma}\label{lem:CumulantDecay}  For $\eps \in (0,1/2)$, let $u$ be a
	rotationally decreasing, 
	symmetric and harmonic function on $B(1,2^{4}\eps)$.
	Let $(a_j)_{j\geq 1}$ be the normalized cumulant sequence of $u$.
	If $(a_j)_{j\geq 2}$ is a non-zero sequence then
	for all $L \geq 2$ we have
	\begin{equation} \label{equ:CumulantFrac} \frac{ \sum_{j\geq L } |a_j|\eps^j }{ \sum_{j\geq 2} |a_j|\eps^i } \leq C \cdot 2^{-L}, \end{equation}
	where $C >0$ is an absolute constant.  
\end{lemma}

We also need the following result from our paper \cite[Lemma 7.1]{clt2}, which says that if there is a small $k$ for which $|a_k|$ is large \emph{then} $|a_2| = 2\Var(X)$ must \emph{also} be large. This lemma can be a seen as an quantitative form of a (non-quantitative) lemma co-discovered by De Angelis \cite{deAngelis1} 
and Bergweiler, Eremenko and Sokal \cite{BES}; further, it may be viewed as a quantitative version of Marcinkiewicz's classical theorem on cumulants \cite{lukacs,marcinkiewicz}:

\begin{lemma}\label{lem:BES} For $s > 0 $ and $L> 2$, let $u$ be a weakly positive, symmetric and harmonic function on $B(1,2s)$ 
	and let $(a_j)_{j}$ be its normalized cumulant sequence. If $(a_j)_{j\geq 2}$
	satisfies
	\begin{equation} \label{equ:cutoff} \sum_{j\geq 2}^{L} |a_j|s^j \geq \sum_{j> L} |a_j|s^j, \end{equation} then there exists a real number $s_{\ast} > s2^{-6(L+1)}$ for which 
	$ |a_2| \geq s^{j-2}_{\ast}|a_j|, $ for all $j \geq 2$.\end{lemma}

\vspace{4mm}

With these tools laid out, we are now in a position to prove the main result of this section. 

\vspace{4mm}

\noindent \emph{Proof of Lemma~\ref{lem:Var-and-phi}.} 
Let $X \in \{0,\ldots,n\}$ be a random variable, let $f_X$ be its probability generating function, let $u = u_X$ be its logarithmic potential and 
let $\vp_{\g} = \vp_{\g,u}$ be the function defined at \eqref{eq:def-of-phi} for $0 < \g < \delta/2^{6}$. Since $f_X$ has no zeros in $S(\delta)$, it follows that 
$u_X$ is harmonic on $S(\delta)$ and thus we may use the expansion of $\vp(e^w)$, 
\[ \vp_{\g}(e^{w}) = \sum_{j \geq 2} a_j\Re\big( w^j - (i\g + w) ^j\big), 
\] as noted in Section~\ref{sec:definitions} at \eqref{eq:phi-series}, which is valid for all  $|w| \leq \delta/2$. Now set $w = 0 $ and apply the triangle inequality to obtain
\[ |\vp_{\g}(1)| \leq \sum_{j \geq 2} |a_j|\g^j. \]
Since $u$ is weakly-positive, symmetric and harmonic in $S(\delta)$ and has logarithmic growth, we may apply Lemma~\ref{lem:decreasing} to see that $u$ is rotationally decreasing in $S(\delta/2)$.  Seeking to apply Lemma~\ref{lem:CumulantDecay} with $\eps = \gamma$, note that $B(1,2^4 \gamma) \subset B(1,\delta/4) \subset S(\delta/2)$ and so for all $L$, we have
\begin{equation} \label{eq:app-of-cumulant-decay} \frac{ \sum_{j\geq L } |a_j|\g^j }{ \sum_{j\geq 2} |a_j|\g^i } \leq C \cdot 2^{-L},
\end{equation} where $C$ is a large, but absolute constant. If we put $L = \log_2 C+1$, we have 
\[ \vp_{\g}(1)/2 \leq \sum_{ j = 2}^{L} |a_j|\g^j,\] and thus by averaging there is a $j \in [L]$ for which 
\begin{equation} \label{eq:aj-and-vp} \vp_{\g}(1)/(2\g^jL) \leq  |a_j|. \end{equation}
Now \eqref{eq:app-of-cumulant-decay} also tells us that
\[\sum_{ j = 2}^{L} |a_j|\g^j \geq  \sum_{ j>L} |a_j|\g^j,\] 
and so we may apply Lemma~\ref{lem:BES} to learn that there is a $s_{\ast} \geq  \g 2^{-6(L+1)} =: \g c_0 $ for which
$|a_2| \geq s_{\ast}^{j-2}|a_j|$. Using this, along with \eqref{eq:aj-and-vp} gives 
\[ \Var(X) = |a_2|/2 \geq \frac{1}{2}(\g c_0 )^{j-2}|a_j| \geq \vp_{\g}(1)(c_0\g)^{-2} (\g c_0)^j\frac{1}{4L\g^j} \geq c\g^{-2}\vp_{\g}(1) , \]
where $c >0$ is an absolute constant. \qed

\section{Connection to a Mellin Transform} \label{sec:connection-to-mellin}

In the previous section, we showed that we could bound $\Var(X)$ in terms of $\vp_{\g}(1)$. 
In this section, we take another step towards the proof of Theorem~\ref{thm:Var-lower-bound},
by obtaining a lower bound for $\vp_{\g}(1)$ in terms of a function\footnote{As with $\vp_{\g}$, we shall usually 
	suppress the explicit dependence on $u$ as it will be clear from context.} $H_{M,\tau}(s) = H_{M,\tau,u}(s)$, which resembles a truncated version of a Mellin transform. 
In particular, for a harmonic function $u$ on $S(0,\tau)$ for $\tau \in (0,\pi)$, $s \in (0,1)$ and $M \geq 1$, we define

\begin{equation}\label{eq:H_R-def} H_{M,\tau,u}(s) := \int_{1/M}^M (u(t) - u(e^{i\tau} t))t^{-(s + 1)}\,dt \, .\end{equation}

We note that taking $M \to \infty$ yields the Mellin transform of our function $\varphi_\tau(t)$.  The following lemma is the main result of this section and allows us to control $\vp_{\g}$ in terms of $H_{M,\tau}$.

\begin{lemma} \label{lem:phi-at-least-H_r} For $\delta \in (0,\pi)$, 
	$M \geq 1$, $s \in (0,1)$ and for $\eta \in (0,1/2)$ set  $\g = \eta \delta$ and let $u$ be a weakly positive harmonic function on $S(\delta)$ that has logarithmic growth. Then  
	\begin{equation} \vp_{\g}(1)  \geq \frac{c_\eta H_{M,\delta}(s)}{\delta M^{2\pi/\delta + s}}, \end{equation}
	where $c_\eta >0$ is a constant depending only on $\eta$.
\end{lemma}

So far we have not made it clear why $H_{M, \delta}$ is any easier to work with than $\vp_{\g}$ and the reader may be skeptical that Lemma~\ref{lem:phi-at-least-H_r} is of any use to us. However, we shall see that $H_{M,\delta}(s)$ has a very different behavior for small $s >0$, which we are able to take full advantage of.  
The reader should also keep in mind that when we apply Lemma~\ref{lem:phi-at-least-H_r} in our proof of Theorem~\ref{thm:Var-lower-bound}, we will ultimately choose $M$ to be a multiple of $R$, $s \to 0$, and $\eta$ an explicit small constant.

We now turn to the proof of Lemma~\ref{lem:phi-at-least-H_r}, which naturally breaks into three steps. In the first step we compare $\vp_{\g}(1)$ with values  $\vp_{\g}(\rho e^{i(\delta-\g)/2})$ for all $\rho \approx 1$. In the second step we use the theory of \emph{Poisson integration} to express $\vp_{\g}(\rho e^{i(\delta-\g)/2})$ in an integral form with positive integrand. Then, finally, we relate the integral form obtained in the second step to the integral $H_{M,\delta}(s)$.

\subsection{Moving away from the boundary}

Here we shall compare $\vp_{\g}(1)$ to the values $\vp_\g(\rho e^{i(\delta - \g)/2})$, where $\rho \approx 1$, by using the connection between harmonic functions and Brownian motion. 
This connection is well-known and we refer the reader 
to Chapters 7 and 8 in the book \cite{peres} for a detailed treatment. 

Here we need a basic estimate on the probability that a Brownian motion hits the top side of a particular polar-rectangle at its first exit. 

\begin{obs}\label{obs:B-motion-hits-side} For $\delta >0$ and $\eta \in (0,1/2)$, let $\g =\eta \delta$, let $(B_t)_t$ be a planar Brownian motion started at $1 \in \C$, let  
	\[ R := \{ \rho e^{i\t} : \rho \in [e^{-\delta},e^\delta], \t \in [-\gamma/2,  (\delta - \g)/2]  \}, \]
	and let $T$ be the stopping time $T := \min\{t : B_t \in \partial R \}$. Then
	\[ \PP\left( \arg(B_T) = (\delta - \g)/2 \right) \geq c_{\eta},\]
	for some constant $c_{\eta} >0$ depending only on $\eta$.
\end{obs}

This estimate together with non-negativity of $\varphi_\gamma$ then allows us to compare $\vp_\gamma(1)$ to values of $\vp$ away from the boundary of $S(0,\delta)$.

\begin{lemma}\label{lem:vp-compare}
	For $\delta \in (0,\pi)$ and for $\eta \in (0,1/2)$, let $\g = \eta \delta$ and let $u$ be a weakly-positive, symmetric, harmonic function on $S(\delta)$
	that has logarithmic growth.
	Then \[ \vp_{\g}(1) \geq c_\eta  \min_{\rho \in[e^{-\delta},e^\delta]} \vp_{\g}(\rho e^{i(\delta-\g)/2}), 
	\]where $c_\eta >0$ is a constant depending only on $\eta$.
\end{lemma}

\begin{proof}
	First note that since $u$ is a weakly-positive, symmetric, harmonic function on $S(\delta)$ with logarithmic growth, then we can apply Lemma~\ref{lem:decreasing} to learn that $u$ is rotationally decreasing in $S(\delta/2)$.  Thus, by Observation~\ref{obs:basic-vp-facts}, we see that $\vp_\g$ is harmonic and non-negative in the region 
	$$R= \{ \rho e^{i\t} : \rho \in [e^{-\delta},e^\delta], \t \in [-\gamma/2,  (\delta - \g)/2]  \}. $$
	We now use the Brownian motion interpretation of harmonic functions to write 
	\begin{equation} \label{eq:vp-boundary} \vp_\g(1) = \E\, \vp_\g(B_T  )\,,     \end{equation}
	where $B_t$ is a standard planar Brownian motion started at $1 \in R$ and $T= \min \{t : B_t \in \partial R\}$. 
	Using Observation~\ref{obs:B-motion-hits-side}, \eqref{eq:vp-boundary} and the fact that $\vp_{\g}(B_T)$ is non-negative in $R$ allows us to bound 
	\[ \vp_\g(1)  
	\geq \PP\left( \arg(B_T) = \frac{\delta - \g}{2} \right) \min_{\rho \in [e^{-\delta},e^{\delta}]} \vp_\g(\rho e^{i(\delta - \g)/2} ) 
	\geq c_{\eta} \min_{\rho \in [e^{-\delta},e^{\delta}]} \vp_\g(\rho e^{i(\delta - \g)/2} ),\] 
	as desired.\end{proof}

\subsection{Obtaining an integral form} 

We now take our second step towards Lemma~\ref{lem:phi-at-least-H_r} by obtaining an integral form for the function $\vp_{\g}$. To state this, let $\tau \in (0,\pi)$, set $g  =  \pi/\tau$, and then for each $z \in S(0,\tau)$ define the function $P_{z,\tau} :\R_{\geq 0} \rightarrow \R_{\geq 0}$ by 
\begin{equation} \label{eq:Def-of-P-Ker} P_{z,\tau}(t) = \frac{g t^{g - 1} \Im(z^g)}{\pi | z^g - t^g|^2}\,. 
\end{equation} The following Lemma gives us our desired integral form of the function $\vp_{\g}$ when evaluated at points of the form $\rho e^{i(\tau - \gamma)/2}$.

\begin{lemma}\label{lem:IntFormForDiff} For $\tau \in (0,\pi)$ and $\g \in (0,\tau/2)$,
	let $u$ be a harmonic function on a neighborhood of $S(0,\tau)$ that has logarithmic growth. Then for $z(\rho) = \rho e^{i(\tau - \g)/2}$ we have 
	\begin{equation} \label{equ:diffofUs} 
		\vp_{\g}(z(\rho)) = \int_{0}^{\infty} ( u(t) - u( te^{i\tau}) )P_{z,\tau/2}(t)\, dt,
	\end{equation} where $P_{z,\tau/2}(t)$ is defined at \eqref{eq:Def-of-P-Ker}.\end{lemma}

Here, we require only two further properties of $P_{z,\tau/2}$, that $P_{z,\tau/2}(t) \geq 0$ for all $t >0$ and the following basic estimate on its growth.

\begin{lemma} \label{lem:Ratio-PKernel-and-ts}
	For $M \geq 1$, $\tau \in (0,\pi)$ and $\eta \in (0,1/2)$ let $\g = \eta \tau$. 
	Then for $s \in (0,1)$ $$\min_{M^{-1} \leq t \leq M}  \min_{e^{-\tau }\leq \rho \leq e^\tau } \{ t^{1 + s}P_{z(\rho),\tau/2}(t)  \} \geq c_{\eta} \tau^{-1} M^{-2\pi/\tau - s}$$
	where $z(\rho) = \rho e^{i(\tau - \g)/2}$ and $c_\eta > 0$ is a constant depending only on $\eta$. \end{lemma}

The idea behind Lemma \ref{lem:IntFormForDiff} is to use the connection between harmonic functions and Brownian motion to write $u(\rho e^{i(\tau - \gamma)/2})$ as an integral over the boundary of $S(0,\tau/2)$ and $u(\rho e^{i(\tau + \gamma)/2})$ as an integral over the boundary of $S(\tau/2,\tau)$.  By symmetry, the contribution of integrals along the ray $\arg(z) = \tau/2$ will precisely cancel when we take the difference $u(\rho e^{i(\tau - \gamma)/2}) - u(\rho e^{i(\tau + \gamma)/2})$, giving the identity Lemma~\ref{lem:IntFormForDiff}.  
We postpone the proofs of Lemmas~\ref{lem:IntFormForDiff} and \ref{lem:Ratio-PKernel-and-ts} to Appendix~\ref{sec:mellin-appendix}, as the details are not particularly interesting and distract somewhat from the main course of our proof. For now, we see how these pieces fit together to prove Lemma~\ref{lem:phi-at-least-H_r}.

\vspace{4mm}
\begin{proof}[Proof of Lemma~\ref{lem:phi-at-least-H_r}]  Let $\tau \in (\delta/2,\delta)$.	For $z = \rho e^{i(\tau - \g)/2}$ we apply Lemma~\ref{lem:IntFormForDiff} and write
	\begin{equation} \label{equ:Poisson-form-of-vp} 
		\vp_\g(z) = \int_{0}^\infty (u(t) - u(te^{i\tau})) P_{z,\tau/2}(t)\,dt\,. \end{equation}
	We now make crucial use of weak-positivity \eqref{eq:weak-positivity} (that is $u(t) - u(te^{i\tau}) \geq 0$) and the fact that $P_{z,\tau/2}(t) \geq 0$, to write 
	\[  \vp_\g(z) \geq \int_{M^{-1}}^M(u(t) - u(te^{i\tau})) P_{z,\tau/2}(t)\,dt .
	\] 
	An application of Lemma~\ref{lem:vp-compare} together with weak-positivity and Lemma~\ref{lem:Ratio-PKernel-and-ts} gives \begin{align*}
		\vp_\g(1) &\geq c_\eta \min_{e^{-\tau} \leq \rho \leq e^\tau} \int_{M^{-1}}^M (u(t) - u(te^{i\tau})) P_{z,\tau/2}(t)\,dt \\
		&\geq c_\eta \min_{M^{-1} \leq t \leq M} \min_{e^{-\tau} \leq \rho \leq e^\tau} \{ t^{1 + s} P_{z,\tau/2}(t) \} \cdot \int_{M^{-1}}^M (u(t) - u(t e^{i\tau})) t^{-(1+ s)}\,dt \\
		&\geq c_{\eta}c'_{\eta} \cdot \tau^{-1} M^{-2\pi/\tau - s} H_{M,\tau}(s)\,.
	\end{align*}Taking $\tau \uparrow \delta$ completes the Lemma.
\end{proof}

\section{A Mellin Transform Calculation} \label{sec:mellin-comp}

The goal of this short section is to compute $H_{M,\tau}(s)$ when $M = +\infty$ and $s \rightarrow 0$. For this, we define
\[ L_{\z,\tau}(t) := 2\log \left|1 - t\z^{-1} \right| - \log \left|1 - e^{i\tau}t \z^{-1} \right| - \log\left| 1 - e^{-i\tau}t \z^{-1} \right|,\]
for $\tau \in (0,\pi)$ and $\z \in \C \setminus \{0\}$, 

Our main goal of this section will be to show the following.

\begin{lemma}\label{lem:whole-log-int} For $\tau \in (0,\pi)$, $\z \in \C \setminus \{0\}$ with $|\arg(\z)| > \tau$ and $s \in (0,1)$, we have 
	\begin{equation} \label{eq:L-whole-int} \int_0^{\infty} L_{\z,\tau}(t) t^{-(s+1)}\, dt = \tau^2+ o_{s \rightarrow 0}(1). \end{equation}
\end{lemma}
This calculation is implicit in the work of Eremenko and Fryntov~\cite{eremenko-fryntov} and 
begins to reveal the special structure that emerges when $s$ is small. Indeed, angular information about the root $\z$ is lost as we send $s \rightarrow 0$. 

We also note the connection between $L_{\z,\tau}$ and our function $H_{M,\tau}$.

\begin{lemma}\label{lem:L-form}  Let $X \in \{0,\ldots,n\}$ be a random variable with probability generating function $f_X$.
	For all $M \geq 1$ and $\tau >0$, we have
	$$H_{M,\tau}(s) = \frac{1}{2} \sum_{\zeta} \int_{1/M}^M L_{\z,\tau}(t) t^{-(s+1)}\, dt ,$$
	where is the sum is over the roots $\{\z\}$ of $f_X$.
\end{lemma}
\begin{proof} Set $u = u_X$ to be the logarithmic potential of $X$. The symmetry property of $u$ implies 
	\[ u(t) - u(e^{i\tau}t) = \frac{1}{2}\left(2u(t) - u(e^{i\tau}) - u(e^{-i\tau})\right),
	\] and then, using the definition of $H_{M,\tau}(s)$, we write 
	\begin{equation} \label{eq:HR-in-bound} H_{M,\tau}(s) = \frac{1}{2}\int_{1/M}^M \left( 2u(t) - u(e^{i\tau}t) - u(e^{-i\tau}t)\right)t^{-(s+1)}\, dt. \end{equation}
	Now, using the expansion of $u$ in terms of its roots (as we noted at \eqref{eq:exp-u-with-roots}) gives
	\[ u(z) = \sum_{|\zeta| < 1} \log\left|1 - \frac{\zeta}{z} \right| + \sum_{|\z| \geq 1 }\log\left|1 - \frac{z}{\zeta}\right| + N_X\log|z| + c_X,
	\] which allows us to write
	\begin{equation}\label{eq:u-diff-as-L-sum}  2u(t) - u(e^{i\tau}t) - u(e^{-i\tau}t) = \sum_{|\z| \geq 1} L_{\z,\tau}(t) + \sum_{|\z| < 1} L_{\z^{-1},\tau}(1/t) = \sum_{\zeta} L_{\z,\tau}(t) 
	\end{equation} by using the identity $L_{\zeta,\tau}(t) = L_{\zeta^{-1},\tau}(1/t)$.  Using \eqref{eq:u-diff-as-L-sum} in \eqref{eq:HR-in-bound} and swapping the sum and integral completes the proof.
\end{proof}

\vspace{2mm}

To aid in our calculation we define (abusing notation slightly), for $\t \in \R$ and $\tau \in (0,\pi)$,
\begin{equation} \label{eq:def-of-Lt} L_{\t,\tau}(t) := 2\log \left|1 - e^{i\theta}t \right| - \log \left|1 - e^{i(\theta + \tau)}t\right| - \log\left| 1 - e^{i(-\theta  +\tau)}t \right|. \end{equation} We also define the function $\phi_s(\t)$ by first setting $\phi_s(\t) = \cos(s(\t - \pi))$, for $\t \in [0,2\pi]$ and then extend this function periodically to all of $\R$. That is, we define $\phi_s(\t) := \cos(s(\t - \pi - 2k\pi))$, for all $\t \in [2\pi k,2\pi(k+1)]$ and all $k \in \Z$. We note the following fact before moving on to the proof of Lemma~\ref{lem:whole-log-int}.
\begin{fact} \label{fact:int-formula}
	For $\theta \in \R$ and $s \in (0,1)$ we have 
	$$\int_{0}^\infty \log\left|1 - e^{i\theta} t \right| t^{-(s+1)}\,dt = c_s \phi_s(\t),   $$
	where $c_s = \pi/ (s \sin(\pi s))$.  
\end{fact}
This identity appears as equation I.4.22 in \cite{mellin} as well as \cite{eremenko-fryntov}.

\begin{proof}[Proof of Lemma~\ref{lem:whole-log-int}] Write $\z = \rho e^{i\t}$. Changing variables, we write
	\begin{equation}\label{eq:chang-var} \int_0^{\infty} L_{\z,\tau}(t) t^{-(s+1)}\, dt  = \rho^{-s}\int_0^{\infty} L_{-\t,\tau}(t) t^{-(s+1)}\, dt
	\end{equation} and then applying Fact~\ref{fact:int-formula} gives
	\[ \int_0^{\infty} L_{-\t,\tau}(t) t^{-(s+1)}\, dt =  c_s\left(2\phi_s( - \t)  - \phi_s(\tau - \t) - \phi_s(\t+\tau)\right). 
	\] Note that since $L_{-\t,\tau} = L_{\t,\tau}$, we may assume that $\t \in [0,\pi]$. So, the periodicity of $\phi_s$ allows us to write the expression in the brackets as 
	\[ 2\phi_s(2\pi - \t)  - \phi_s(2\pi + \tau - \t) - \phi_s(\tau + \t)  
	\] which is equal to
	\begin{equation}\label{eq:mellin-L-cos} 2\cos(s(\pi - \t))  - \cos(s(\tau - \t + \pi)) - \cos(s( \tau + \t -\pi)), \end{equation}
	by the definition of $\phi_s$ and the fact that $\t \in [\tau,\pi]$, which implies that each of the arguments 
	\[ 2\pi - \t , 2\pi + \tau - \t, \tau + \t \in [0,2\pi].
	\] We now use the Taylor expansion of cosine to express \eqref{eq:mellin-L-cos} as 
	\[  s^2\tau^2  - \frac{s^4 \tau^2}{12}(6(\pi - \tau)^2 + \tau^2) + {E}_{\zeta}(s) = s^2\tau^2 + O(s^4).
	\] where $|{E}_\zeta(s)| \leq \frac{4}{6!}((s\pi)^6)$ for $s < 1$.  We now notice that 
	\[ \lim_{s \rightarrow 0} s^2c_s = \lim_{s\rightarrow 0} \frac{\pi s^2}{s \sin(\pi s)} = 1 \] and thus
	\[ \int_0^{\infty} L_{\t,t}(t) t^{-(s+1)}\, dt =  c_s\left(s^2\tau^2 + O(s^4)\right) = \tau^2 + o_{s\rightarrow 0}(1) .
	\] Putting this together with \eqref{eq:chang-var} finishes the proof for $\rho \geq 1$ after noting that $\rho^{-s} = 1 + o_{s \to 1}(1)$. The proof for $\rho\leq 1 $ is symmetric. 
\end{proof}

\section{Truncating the Mellin Transform}\label{sec:truncation}

In this section we prove the following lemma which will be essential to obtaining a lower bound on $H_{M,\tau}(s)$.

\begin{lemma} \label{lem:mainLogBound} For $\eps \in (0,1/2)$, $M \geq 1+\eps$ and $\tau \in (0,\pi)$, let $\z \in \C \setminus \{0\}$ 
	satisfy $\tau < |\arg(z)| \leq \pi$ and $M^{-1} \leq |\z| \leq M$. Then, for sufficiently small $s >0 $, we have
	\begin{equation} \label{equ:trunc-geq0} \int_{1/M}^M L_{\z,\tau}(t)t^{-(s+1)}\, dt \geq 0. \end{equation}
	If we additionally have $|\arg(\z)| \geq \pi/4$ then, for sufficiently small $s >0$, we have
	\begin{equation} \label{eq:trunc-int} \int_{1/M}^M L_{\z,\tau}(t)t^{-(s+1)}\, dt \geq  |\zeta|^{-s}2^{-10}\eps \tau^2 (1 + o_{\tau \rightarrow 0}(1)). \end{equation}
\end{lemma}

We will also need the following cheap bound to deal with the case when $\delta$ is bounded away from $0$.

\begin{lemma} \label{lem:mainLogBound2}
	For $\alpha \in (0,1)$ and $\tau \in (0,\pi)$ let $\z \in \C \setminus \{ 0 \}$, be such that $(\alpha M)^{-1} \leq |\z| \leq \alpha M$ and $|\arg(\zeta)| > \tau$. Then we have 
	\[ \int_{1/M}^M L_{\z}(t)t^{-(s+1)}\, dt  \geq |\zeta|^{-s} \left( \tau^2/2  + O\left( \alpha^{1-s} \right)  \right),\] for sufficiently small $s >0$. 
\end{lemma}

We note that this bound is sufficient to prove a weaker version of Theorem \ref{thm:Var-lower-bound} of the form $\Var(X) \geq c_\delta R^{-2\pi/\delta} n$ if one does not care about the dependence on $\delta$.

\subsection{A few preparations}
To prove Lemmas~\ref{lem:mainLogBound} and \ref{lem:mainLogBound2} we need a few useful results about the family of functions  $L_{\theta,\tau}$ first:

\begin{obs} \label{obs:L-stuff} Let $\t\in [-\pi,\pi]$ and $\tau \in (0,\pi)$.
	\begin{enumerate}
		\item \label{obs:symmetry} For $t>0$, we have $L_{\t,\tau}(t) = L_{\t,\tau}(1/t)$; 
		\item \label{obs:decay} $|L_{\t,\tau}(t)| = O(1/t)$ as $t\rightarrow \infty$ and $|L_{\t,\tau}(t)| = O(t)$ as $t \rightarrow 0$.
	\end{enumerate}
\end{obs}

We now use the symmetry of $L_{\t,\tau}(t)$ to obtain an estimate for a truncated version of the integral featured in Lemma~\ref{lem:whole-log-int}.

\begin{lemma}\label{lem:int-0-to-1} For $\tau\in (0,\pi)$, $\tau < |\t| \leq \pi$ and $s \in (0,1)$, we have that 
	\[ \int_0^{1} L_{\t,\tau}(t)t^{-(s+1)}\, dt = \tau^2/2 + o_{s\rightarrow 0}(1). \]
\end{lemma}
\begin{proof}
	Write $L(t) =  L_{\t,\tau}(t)$ and set
	\[ I_0 := \int_0^1 L(t) t^{-(s+1)}\, dt \quad\text{ and }\quad I_{\infty} := \int_1^{\infty}   L(t) t^{-(s+1)} \, dt.
	\] Applying Lemma~\ref{lem:whole-log-int} to $\z = e^{-i\t}$, we have 
	\begin{equation} \label{eq:sum-I0-Iinfty} I_0 + I_{\infty} = \int_0^{\infty} L(t) t^{-(s+1)} \, dt = \tau^2 + o_{s\rightarrow 0}(1).
	\end{equation} Now note that by changing variables $t = 1/x$ and using the symmetry $L(1/x) = L(x)$ (Observation~\ref{obs:symmetry}), we have 
	\[ I_0 = \int_1^{\infty} \frac{L(x)}{x^{1+s}} x^{2s} \, dx
	\] and thus we can see that $I_0 \approx I_{\infty}$ for small $s$. Indeed, using the fact that $L(t) = O(1/t)$ (Observation~\ref{obs:decay})
	we have
	\begin{equation} \label{eq:diff-bound} |I_0 - I_{\infty}| = \left| \int_1^{\infty} \frac{L(t)}{t^{s+1}}(t^{2s}-1) \, dt \right| 
		\leq C \int_{1}^{\infty} \frac{1}{t^{s+2}}(t^{2s}-1)\, dt, \end{equation}
	which tends to $0$ as $s \rightarrow 0$. To finish, note \eqref{eq:sum-I0-Iinfty} tells us that $I_0 + I_{\infty} \approx \tau^2$. Rearranging this gives  
	\[ |2I_0 - \tau^2| \leq \tau^2 + |I_0- I_{\infty}| + o_{s\rightarrow 0}(1) = o_{s\rightarrow 0}(1),
	\] as desired.
\end{proof}

\vspace{4mm}

The following elementary lemmas will allow us to throw away parts of the integral that are negative. 
First we note that $L_{\theta,\tau}$ undergoes at most one sign-change on $[1,\infty)$.  
When working with the sign-change $\l(\t,\tau)$, which we define in the following lemma, we subscribe to the convention that $1/\infty = 0$.

\begin{lemma} \label{lem:sign-change} Let $\tau\in (0,\pi)$.
	\begin{enumerate}
		\item \label{part:theta-small} If $\t \in (\tau,\pi/2)$, there exists a sign-change $\l = \l(\t,\tau) > 1$ for which 
		$L_{\t,\tau}(t) \geq 0$ for $t \in [1,\l]$ and $L_{\t,\tau}(t) \leq 0$, for $t \geq \l$; 
		\item \label{part:theta-big} If $\t  \in [\pi/2,\pi)$ then $L_{\t,\tau}(t) \geq 0$ for all $t > 0$. In this case, we define $\l(\t,\tau) := +\infty$.
	\end{enumerate}
\end{lemma}

We now use this sign-change to lower bound the contribution of the integral near $1$ in the case of $|\theta| \geq \pi/4$:

\begin{lemma}\label{lem:small-log-bound} 
	For $\tau \in (0,\pi)$ and $\t$ satisfying $\pi/4 < |\t| \leq \pi$, let $\eps \in (0,1/2)$ be such that $\l^{-1} \leq 1-\eps$, where $\l = \l(\t,\tau)$ is the sign change of $L_{\t,\tau}$. Then
	\[ \int_{1-\eps}^1 L_{\t,\tau}(t)t^{-(s+1)} \, dt \geq \frac{\eps \tau^2 }{2^{9}}(1 + o_{\tau \rightarrow 0}(1)). \] 
\end{lemma}

As the proofs of Lemmas~\ref{lem:sign-change} and Lemma~\ref{lem:small-log-bound} are somewhat tedious, we postpone them to Appendix~\ref{sec:trunc-appendix}.

\subsection{Proofs of Lemmas~\ref{lem:mainLogBound} and ~\ref{lem:mainLogBound2}} 
We now use Lemmas~\ref{lem:int-0-to-1}, \ref{lem:sign-change} and \ref{lem:small-log-bound} to prove Lemma~\ref{lem:mainLogBound},
the main result of this section. 

\vspace{4mm}

\begin{proof}[Proof of Lemma~\ref{lem:mainLogBound}]
	Let $\tau \in (0,\pi)$ and let $\z = \rho e^{i\t}$ for $\rho >0$ and $\t$ satisfying $\tau < |\t| \leq \pi$.
	We first assume $\rho \geq 1$ and note that
	\begin{equation} \label{eq:trunc-change-of-var} \int_{1/M}^M L_{\z}(t) t^{-(s+1)}\, dt = \rho^{-s} \int_{1/(M\rho)}^{M/\rho} L_{\t,\tau}(x) x^{-(s+1)}\, dx,\end{equation}
	by the change of variables $x = t/M$. We put $a := 1/(r\rho)$, $b :=r/\rho$, $L(t) := L_{\t,\tau}$ and 
	let $\l := \l(\t,\tau)$ be the sign-change from Lemma~\ref{lem:sign-change}. We proceed in two different cases depending on whether $b \geq \l$ or $b < \l$. 
	
	\noindent \textbf{Case 1} : Assume $b \geq \l$.
	Then $L(t) \leq 0$ for all $t \geq b$ and since 
	\[ a = 1/(M\rho) \leq \rho/M = 1/b \leq 1/\l\] we have that 
	$L(t) \leq 0$ for all $0 \leq t \leq a$. As a result, we have 
	\[ \int_{a}^{b} L(t) t^{-(s+1)}\, dt \geq \int_{0}^{\infty} L(t) t^{-(s+1)}\, dt \geq \tau^2/2, 
	\] where the last inequality holds by \eqref{eq:L-whole-int}, in Lemma~\ref{lem:whole-log-int}, for sufficiently small $s$.
	
	\noindent \textbf{Case 2} : Assume now that we have $1 \leq b \leq \l$. In this case we have that $L(t) \geq 0$ for all $t \in [1,b]$. We now break into two further
	cases. If $a \leq \l^{-1}$, then for sufficiently small $s$ we have 
	\begin{equation} \label{eq:using-0-1-int} 
		\int_{a}^{b}  L(t) t^{-(s+1)} \, dt \geq \int_{0}^{b} L(t) t^{-(s+1)} \, dt  \geq \int_{0}^{1} L(t) t^{-(s+1)} \, dt  \geq \tau^2/4,\end{equation}
	where the second inequality follows from the fact that $\l \geq b \geq 1$.
	
	In the final case we have $a \geq \l^{-1}$ and hence $L(t) \geq 0$ for all $t \in [a,b]$, thus we have that 
	\[ \int_{a}^{b}  L(t) t^{-(s+1)} \, dt \geq 0.\] 
	In the case that $|\t| \geq \pi/4$, we have to additionally show \eqref{eq:trunc-int}.
	Since $r \geq 1+\eps$, we have $a \leq (1+\eps)^{-1} \leq 1 - \eps + \eps^2 \leq 1-\eps/2$, since $0 <\eps < 1/2$. Thus
	\[ \int_{a}^{b}  L(t) t^{-(s+1)} \, dt \geq \int^1_{1-\eps/2} L(t) t^{-(s+1)} \, dt \geq\frac{\eps \tau^2 }{2^{10}}(1 + o_{\tau \rightarrow 0}(1)), 
	\] by using Lemma~\ref{lem:small-log-bound}. This completes the proof when $\rho \geq 1$. The case $\rho \leq 1$ is symmetric.
\end{proof}

The proof of Lemma~\ref{lem:mainLogBound2} is a simple application of Observation~\ref{obs:decay} and Lemma~\ref{lem:whole-log-int}.

\begin{proof}[Proof of Lemma~\ref{lem:mainLogBound2}]
	Let $s\in (0,1)$, which we will take to be sufficiently small and write $\zeta = \rho e^{i\t}$.  Then we have
	\[ \int_{1/M}^M L_{\z,\tau}(t) t^{-(s+1)}\, dt =  \rho^{-s}\int_{1/(M\rho)}^{M/\rho} L_{-\t,\tau}(t) t^{-(s+1)}\, dt. \]
	Put $L(t) := L_{-\t,\tau}(t)$ and define $I_1 := \int_0^{1/(M\rho)} L(t) t^{-(s+1)}\,dt$ and $I_2 :=\int_{M/\rho}^{\infty} L(t) t^{-(s+1)}\,dt$ so that
	\[ \int_{1/(M \rho)}^{M/\rho} L(t) t^{-(s+1)}\, dt = \int_0^{\infty} L(t)t^{-(s+1)}\, dt - I_1 - I_2 = \tau^2 - I_1 - I_2 + o_{s \rightarrow 0}(1).\]
	We now use that $L(t) = O(t)$, for $t$ small, and the fact that $1/(M\rho) \leq \alpha$ to bound 
	\[ |I_1| = \left|\int_0^{1/(M\rho)} L(t)t^{-(s+1)}\, dt\right| \leq  C\int_0^{1/(\rho M)} t^{-s}\, dt = O\left({\alpha^{1-s}} \right).   \] 
	Similarly, since $|L(t)| = O(1/t)$ for $t$ large, and $\rho/M \leq \alpha $, we have 
	\[ |I_2| = \left|\int_{M/\rho}^{\infty} L(t)t^{-(s+1)}\, dt\right| \leq C \int_{M/\rho}^{\infty}t^{-(2+s)} \, dt = O\left(\alpha^{1+s}\right).  \] 
	Thus, for $s >0$ sufficiently small, the result follows from Lemma \ref{lem:whole-log-int}.
\end{proof}

\section{The Proof of Theorem~\ref{thm:Var-lower-bound} } \label{sec:proof}

We now arrive, at last, at the proof of Theorem~\ref{thm:Var-lower-bound}.  As we mentioned in the introduction, we actually prove a slightly sharper result.

\begin{theorem}\label{thm:Var-lower-bound-sharper} Let $X \in \{0,\ldots,n\}$ be a random variable with $p_0p_n >0$ and with probability generating function $f_X$
	and let  $\eps \in (0,1), \delta >0, R \geq 1 +\eps$. If the zeros $\z$ of $f_X$ satisfy $ |\arg(\z)| \geq \delta$ and $R^{-1} \leq |\z| \leq R$ then
	\[ \Var(X) \geq c \max\{\eps,\delta\} \cdot \delta^{-1} R^{-2\pi/\delta} n,\]
	where $c>0$ is an absolute constant.  
\end{theorem}

One last ingredient is the classical theorem of Obrechkoff \cite{obrechkoff} on the radial distribution of the set of zeros of a polynomial with non-negative coefficients.

\begin{theorem} \label{thm:Os-thm} Let $f(z)$ be a polynomial with non-negative coefficients for which $f(0) \not= 0$. Then 
	\[ \left| \left\lbrace \z\in \C : f(\z) = 0, |\arg(\z)| \leq \alpha \right\rbrace \right| \leq \frac{2\alpha}{\pi}\deg(f). \]
\end{theorem}

\begin{proof}[Proof of Theorem~\ref{thm:Var-lower-bound-sharper}]
	
	Let $u = u_X$ be the logarithmic potential of $X$.  From the discussion in Section~\ref{sec:definitions}, 
	we know $u$ to be weakly-positive, symmetric and harmonic in 
	\[ S(\delta) = \{ z : |\arg(z)| \leq \delta \}\,.\] 
	We also know that $u$ has logarithmic growth. 
	
	Note that we may assume that $n \geq 1$, as the case $n=0$ leaves us with nothing to prove. 
	Now, looking to apply Lemma~\ref{lem:Var-and-phi}, we note that $\Var(X) >0$, since  $\PP(X=0)\PP(X=n)>0$. Choose $\g =\delta/2^{7}$ and apply Lemma~\ref{lem:Var-and-phi} to carry out our first step and obtain
	\begin{equation}\label{eq:final1} \Var(X) \geq c_1\g^{-2}\vp_{\g}(1).
	\end{equation}
	
	We now introduce a parameter $M = \alpha^{-1}R$ for $\alpha \in (0,1)$ to be chosen later.   Since $u$ is a weakly positive harmonic function on $S(\delta)$ 
	with logarithmic growth, we may apply Lemma~\ref{lem:phi-at-least-H_r} with parameter $M$, $\tau \in (\delta/2,\delta)$ and $\eta = 2^{-7}$ to obtain 
	
	\begin{equation}\label{eq:final2} \vp_{\g}(1) \geq  \frac{c_2 H_{M,\tau}(s)}{\tau M^{2\pi/\tau + s}},\end{equation}
	for all $s \in (0,1)$.  
	
	We now use Lemma~\ref{lem:L-form} and \eqref{equ:trunc-geq0} in Lemma~\ref{lem:mainLogBound}, for sufficiently small $s$, to remove all terms with $|\arg(\z)| \leq \pi/4$ from the sum. Indeed, we have
	\begin{equation}\label{eq:final2.5} H_{M,\tau}(s) = \frac{1}{2}\sum_{\z} \int_{1/M}^M L_{\z,\tau}(t) t^{-(s+1)} \, dt \geq \frac{1}{2}\sum_{\z : |\arg(\z)| \geq \pi/4} \int_{1/M}^M L_{\z,\tau}(t) t^{-(s+1)}. \end{equation} 
	There are now a few cases to consider. We let $\delta_0 >0$ be a (small) absolute constant (to be chosen later) and split into two cases, depending on whether $\delta > \delta_0$ or not.

	{ \bf Case 1 :} We first consider an easy case, when $\delta > \delta_0$. In this case, we consider $\tau \in (\delta_0, \delta)$
	and set $M := \alpha^{-1} R$, where $\alpha \in (0,1)$ is chosen to be a sufficiently small constant so that when we apply Lemma~\ref{lem:mainLogBound2} we get
	\begin{equation} \label{eq:2.75} \int_{1/M}^M L_{\z}(t)t^{-(s+1)}\, dt  \geq R^{-s}\left( \tau^2/2  + O\left( \alpha^{1-s} \right)\right) \geq R^{-s}\tau^2/4,  \end{equation}
	for sufficiently small $s$.
	So, with this choice, apply \eqref{eq:2.75} to \eqref{eq:final2.5} to obtain
	\begin{equation} \label{eq:2.99} H_{M,\tau}(s) \geq R^{-s}\frac{\tau^2}{8} \left(\sum_{\z : |\arg(\z)| \geq \pi/4} 1 \right) \geq  \frac{\tau^2n}{32}, 
	\end{equation} for sufficiently small $s>0$, where the second inequality follows from Obrechkoff's theorem, Theorem~\ref{thm:Os-thm}.  Chain together \eqref{eq:final1} and \eqref{eq:final2} with \eqref{eq:2.99} to obtain
	\[ \Var(X) \geq \frac{c_1}{\g^2}\vp_{\g}(1) \geq \frac{c_3}{\g^2\tau} H_{M,\tau}(s)M^{-2\pi/\tau + s} \geq \alpha^{-s}\alpha^{2\pi/\tau}\frac{c_3\tau}{32\g^2}R^{-2\pi/\tau}n. \] 
	Thus recalling that $\g = \tau/2^{7}$ and that $\alpha^{\pi/\tau} \geq \alpha^{\pi/\delta_0} >0$ is a constant, we may take $\tau \uparrow \delta$ to complete 
	the proof.
	
	{ \bf Case 2 :} In this case we may assume that $\delta$ satisfies $\delta_0 > \delta >0$, where $\delta_0$ is an absolute constant to be chosen later. 
	We choose $M = (1+\delta)R$ and apply \eqref{eq:trunc-int} from Lemma~\ref{lem:mainLogBound} to each term in the sum in \eqref{eq:final2.5} to get
	\begin{equation} \label{eq:final3} 
		H_{M,\tau}(s) 
		\geq  
		R^{-s}\left(\frac{\eps_0 \tau^2 }{2^{11}}(1 + o_{\tau \rightarrow 0}(1))\right)\left( \sum_{\z : |\arg(\z)| \geq \pi/4} 1  \right)
		\geq \left(R^{-s}\frac{\eps_0 \tau^2 }{2^{12}}(1 + o_{\tau \rightarrow 0}(1))\right)n,\end{equation}
	where $\eps_0 = \max\{ \eps, \delta \}$. We may now assume that $\delta_0$ is small enough so that $0 < \tau < \delta_0$ implies that $(1+o_{\tau \rightarrow 0}(1)) \geq 1/2$ in \eqref{eq:final3}.
	Thus 
	\[H_{M,\tau}(s) \geq \frac{\eps_0 \tau^2}{2^{14}}n \] and 
	\[ \Var(X) \geq \frac{c_1}{\g^2}\vp_{\g}(1) \geq \frac{c_2}{\g^2\tau} H_{M,\tau}(s)M^{-2\pi/\tau + s} \geq c_3(1+\delta)^{-2\pi/\delta} \eps_0\tau^{-1}R^{-2\pi/\tau} n,\] for sufficiently small $s>0$. Since $(1+\delta)^{-2\pi/\delta}$ is bounded below, the result follows by taking $\tau \uparrow \delta$. This completes the proof of Theorem~\ref{thm:Var-lower-bound}.

\end{proof}
%%% AUTHOR: optional appendix here
\appendix %% you may comment this out if no Appendix
\section*{Appendix}
\section{Details from Section~\ref{sec:connection-to-mellin}} \label{sec:mellin-appendix}

We will require a few facts about \emph{Poisson kernels}, which we interpret as the probability density function of the location of where Brownian motion exits a region.
For this, define 
\[ \HH := \{z \in \C: \Im(z) > 0 \}. \]

\begin{fact}\label{fact:UHP}
	Let $z \in \HH$, $(B_t)_{t\geq 0}$ be a Brownian motion in $\HH$, started at $z$ and let $T$ be the stopping time $T = \inf\{t : \Im(B_t) = 0 \}$.  Then the density of the random variable $B_T \in \R$ is $$ P_{\HH,z}(s) = \frac{\Im(z)}{\pi |z - s|^2}\,.$$
\end{fact}
\begin{proof}
	For bounded domains, the fact that the Poisson kernel is the hitting density of Brownian motion follows from \cite[Theorem 8.5]{peres} together with uniqueness of solutions to the Dirichlet problem on bounded domains \cite[Chapter 1]{harmonic-axler}.  By conformal invariance of Brownian motion \cite[Theorem 7.20]{peres}, this means that the Poisson kernel is the hitting density of Brownian motion for the upper half-plane.  The Poisson kernel for the upper-half plane is then given in \cite[Chapter 7]{harmonic-axler}.
\end{proof}

\begin{corollary} \label{cor:poisson-sector}
	For $\delta > 0$ and $z \in S(0,\delta)$, let $(B_t)_{t\geq 0}$ be a Brownian motion in $\C$ started at $z$ and let $T = \inf\{ t \geq 0 : B_t \in \partial S(0,\delta) \}$.  Then the density of the random variable $B_T$ is given by 
	$$P_{z,\delta}(s) = \frac{g s^{g - 1} \Im(z^g)}{\pi |z^g - s^g|^2}, \qquad P_{z,\delta}(e^{i\delta} s) = \frac{g s^{g - 1} \Im(z^g)}{\pi |z^g + s^g|^2} $$
	where $g = \pi/\delta$.
\end{corollary}
\begin{proof}
	For a point $0 < s \in \partial S(0,\delta)$, we will differentiate $\PP(B(T) \in [s,s+h])$ with respect to $h$ and evaluate at $h = 0$.  Note that the map $z \mapsto z^g$ is conformal on $S(0,\delta)$ and maps $S(0,\delta)$ to the $\mathbb{H}$.  By conformal invariance of planar Brownian motion---e.g.\ \cite[Theorem 7.20]{peres}---we have that $$\PP(B(T) \in [s,s+h]) = \PP(\tilde{B}(\tilde{T}) \in [s^{g},(s+h)^{g}] )$$ where $\tilde{B}$ is a Brownian motion started at $z^{g}$ and $\tilde{T} = \inf\{t \geq 0 : \tilde{B}(t) \in \partial \mathbb{H}  \}$.   By Fact \ref{fact:UHP}, we have $$\PP(\tilde{B}(\tilde{T}) \in [s^{g},(s+h)^{g}] ) = \int_{s^{g}}^{(s+h)^{g}}\frac{\Im(z^{g})}{\pi |z^{g} - x|^2}\,dx\,.  $$
	Differentiating with respect to $h$ gives $$P_{z,\delta}(s) = \frac{d}{dh} \left( \int_{s^{g}}^{(s+h)^{g}}\frac{\Im(z^{g})}{\pi |z^{g} - x|^2}\,dx\right)\Bigg|_{h = 0} = \frac{g s^{g - 1} \Im(z^g)}{\pi |z^g - s^g|^2}\,.$$
	Noting that $P_{z,\delta}(s) = P_{e^{i\delta}\zbar,\delta}(e^{i\delta} s)$ yields $P_{z,\delta}(e^{i\delta}s) = P_{e^{i\delta} \zbar,\delta}(s)$, thereby giving 
	$$P_{z,\delta}(e^{i\delta}s) = \frac{g s^{g -1} \Im(- \zbar^g)}{\pi| -\zbar^g - s^g|^2} = \frac{g s^{g - 1} \Im(z^g) }{\pi |z^g + s^g|^2}\,.$$
\end{proof}

Our goal is to prove Lemma~\ref{lem:IntFormForDiff}.  We first prove a more general statement about Poisson integration on unbounded regions.

\begin{lemma}\label{lem:poisson-integration-sector}
	Let $u$ be a harmonic function on a neighborhood of $S(0,\delta)$ with 
	$$|u(z)| = O(\log|z|) \text{ as } z \to \infty \text{ and } |u(z)| = O(\log |z|^{-1} ) \textit{ as } z \to 0,$$
	while $z \in S(0,\delta)$.  	Then $$u(z) = \int_{w \in \partial S(0,\delta)} P_{z,\delta}(w) u(w)\,dw\,.$$ 
\end{lemma}

We first require a lemma about the probability that Brownian motion exits a cone before leaving a disk, which we prove in \cite{clt2}.
For this, define $S_R(0,\delta) = \{ z \in S(0,\delta) : R^{-1} \leq |z| \leq R \}$ and $S_R^{\ast}(0,\delta) := \{ z \in \partial S_R(0,\delta) : |z| \in \{R,R^{-1} \}\}$.

\begin{lemma}[Lemma 4.3 of \cite{clt2}] \label{lem:upper-bound}
	For $1 < r < R$ and $\delta > 0$, let $B(t)$ be planar Brownian motion started at some point $z \in S(0,\delta)$ with $|z| = r$ and set $T = \inf\{t : B(t) \in \partial S_R(0,\delta) \}$.  Then there exist constants $C,c > 0$ so that $$\PP(B(T) \in S_R^\ast(0,\delta) ) \leq C \left(\frac{r}{R}\right)^{-c/\delta}\,.$$
\end{lemma}

We can now prove Lemma~\ref{lem:poisson-integration-sector}, the basic fact about Poisson integration on a sector. 

\begin{proof}[Proof of Lemma~\ref{lem:poisson-integration-sector}]
	Let $z \in S(\delta)$, put $S_R := S_R(0,\delta)$ and $S^{\ast}_R := S_R^{\ast}(0,\delta)$. Of course, for all large $R$, we have $z \in S_R$. Now, letting 
	$P_{z,\delta}^R(w)$ denote the Poisson kernel of $S_{R}$, we have $$u(z) = \int_{w \in \p S_R} P_{z,\delta}^{R}(w) u(w)\,dw\,.$$
	
	We partition $\p S_R = S_R^\ast \cup E_R$ and write $$u(z) = \int_{w \in E_R} P_{z,\delta}^{R}(w)u(w)\,dw + \int_{w \in S_R^\ast} P_{z,\delta}^R(w) u(w)\,dw\,.$$
	
	Now, Lemma~\ref{lem:upper-bound} implies that $\int_{w \in S_R^\ast} P_{z,\delta}^R(w) \,dw = O(R^{-c/\delta})$ as $R \to \infty$.  The growth assumption on $u(w)$ then assures that the latter integral converges to $0$ as $R \to \infty$.  Further, note that $P_{z,\delta}^R(w)$ is increasing in $R$, and converges to $P_{z,\delta}(w)$, the Poisson kernel of $S(0,\delta)$.  Thus, we may take $R \to \infty$ in the above to complete the proof.
\end{proof}

\begin{proof}[Proof of Lemma~\ref{lem:IntFormForDiff}]
	For simplicity, write $z_1 = z_{\tau,\gamma}$, $z_2 = e^{i\gamma}z_1$ and $P_z(\cdot)$ for the Poisson kernel of $S(0,\tau/2)$.  Applying Lemma \ref{lem:poisson-integration-sector} to $S(0,\tau/2)$, we have 
	\begin{equation}\label{eq:uz1-expression}
		u(z_1) = \int_{0}^\infty \left( P_{z_1}(t) u(t)  + P_{z_1}(e^{i\tau/2} t) u(e^{i\tau/2}t)\right)\,dt\,.\end{equation}
	
	We may apply Lemma \ref{lem:poisson-integration-sector} to $S(\tau/2,\tau)$ by rotating clockwise by $\tau/2$; this amounts to multiplying all factors of $z$ by $e^{-i\tau/2}$ throughout the Poisson kernels.  We then obtain 
	$$ u(z_2) = \int_0^\infty \left( P_{e^{-i\tau/2}z_2}(t) u(e^{i\tau/2}t)  + P_{e^{-i\tau/2} z_2}(e^{i\tau/2} t) u(e^{i\tau}t)\right)\,dt\,.  $$
	
	We now make use of two crucial identities: as noted in the proof of Corollary \ref{cor:poisson-sector}, we have $P_z(t) = P_{e^{i\tau/2}\zbar}(e^{i\tau/2} t)$; by the choice of $z_1$ and $z_2$, we have $\overline{e^{-i\tau/2}z_2} e^{i\tau/2} = z_1$.  This simplifies the expression for $u(z_2)$: \begin{equation}\label{eq:uz2-expression}
		u(z_2) = \int_0^\infty \left(P_{z_1}(e^{i\tau/2}t) u(e^{i\tau/2}t) + P_{z_1}(t)u(e^{i\tau}t)\right)\,dt\,.\end{equation}
	
	Subtracting \eqref{eq:uz2-expression} from \eqref{eq:uz1-expression} yields $$\vp_\gamma(z_{\tau,\gamma}) = u(z_1) - u(z_2) = \int_0^\infty P_{z_1}(t)(u(t) - u(e^{i\tau t}))\,dt\,.$$
\end{proof}

\begin{proof}[Proof of Lemma~\ref{lem:Ratio-PKernel-and-ts}]
	
	Set $g = \pi/(\tau/2) = 2\pi/\tau$.  Note that for $t > 0$ and $\sigma \in [e^{-\pi},e^{\pi}]$ we have that $P_{z,\tau/2}(t) = \s^{-1} P_{z/\s,\tau/2}(t/\s) \geq e^{-\pi} P_{z/\s}(t/\s)$  and thus it is sufficient to find a lower bound in the case of $z = e^{i(\tau - \gamma)/2}$ for $(M e^\tau)^{-1} \leq t \leq M e^\tau$. 
	
	Observe that 
	$$\Im(z^g) 
	= \Im\left( e^{i \pi}\cdot e^{-i\gamma \pi /\tau} \right) = \sin(\pi \gamma/\tau)  = \sin(\pi \eta)\,.$$
	For $t \in [1,Me^\tau ]$, we then have $$
	t^{1+s} P_{\tau,\tau/2}(t) = 2\sin(\pi \eta) \frac{t^{g +s}}{\tau |e^{i(\tau - \gamma)/2} - t^g|^2  } \geq  \sin(\pi \eta)\frac{t^{-g + s}}{2\tau} \geq c_1 \sin(\pi \eta)  \tau^{-1} M^{-2\pi/\tau + s}
	$$
	where $c_1 > 0$ is an absolute constant.
	
	For $t \in [(Me^\tau)^{-1},1]$, we have $$
	t^{1+s} P_{\tau,\tau/2}(t) = 2\sin(\pi \eta) \frac{t^{g +s}}{\tau |e^{i(\tau - \gamma)/2} - t^g|^2  } \geq c_\eta \tau^{-1} t^{g+s} \geq c_{\eta}' \tau^{-1} M^{-2\pi/\tau - s}\,.$$
\end{proof}

\section{Calculations from Section~\ref{sec:truncation}} \label{sec:trunc-appendix}

For our discussion here we write $L(t) = L_{\t,\tau}(t)$. We now look to write express $L(t) = \frac{1}{2}\log(1 - \alpha(t))$. Using the identity
\[ \log|1 - e^{i\t}t | = \frac{1}{2} \log(1 - 2t \cos(\theta) + t^2)
\] we may obtain the expression
\[ L(t) = \frac{1}{2}\log \left( \frac{(1 - 2t \cos(\theta) + t^2)^2}{(1 - 2t \cos(\theta + \tau) + t^2)(1 - 2t\cos(\theta - \tau ) + t^2)} \right). \]
This allows us to write $L(t)= \frac{1}{2}\log(1 - \alpha(t) )$ where $\alpha(t) = A(t)/B(t)$ and 
\begin{equation} \label{equ:defA} A = 4(1- \cos \tau)(\cos \t ) t - 4(\sin \tau)^2t^2 + 4(1-\cos \tau )(\cos \t) t^3; \end{equation}
\begin{equation} \label{eqy:defB}
	B = (1 - 2t \cos(\theta + \tau) + t^2)(1 - 2t\cos(\theta - \tau ) + t^2).\end{equation} We may now prove Lemma~\ref{lem:sign-change}.

\begin{proof}[Proof of Lemma~\ref{lem:sign-change}]
	From the discussion above, we may write, for all $t >0$, 
	\[ L(t) = \frac{1}{2}\log\left( 1 - \alpha(t) \right),
	\] where $\alpha(t) = A(t)/B(t)$ and $A(t),B(t)$ are real-valued polynomials in $t$. Since 
	\[ B(t) = |1-te^{i(\t+\tau)}|^2|1-te^{i(\t-\tau)}|^2 \geq 0, 
	\] the sign of $L(t)$ is completely determined by the sign of $A(t)$. 
	We rewrite 
	\[ A(t) = t\left(4(1- \cos \tau)(\cos \t ) - 4(\sin \tau)^2t + 4(1-\cos \tau )(\cos \t) t^2 \right),
	\] by first defining $\mu = \frac{(\sin \tau )^2}{(1-\cos \tau)(\cos \t)}$ in the case $\t \not= \pi/2$, so that 
	\[ A(t) = 4\left(1-\cos \tau )(\cos \t \right)t( 1 - \mu t + t^2). 
	\] Note that 
	\[ |\mu|\geq \frac{(\sin \tau)^2 }{(1-\cos \tau)|\cos \tau|} \geq \min_{x \in [0,1]} \frac{1 - x^2}{x-x^2} \geq 2 \]
	and thus the polynomial $P(t) := (1-\mu t + t^2)$ has two real roots $\frac{\mu}{2} \pm \frac{1}{2}\sqrt{\mu^2- 4}$. 
	If $\t \in (\tau,\pi/2)$ then $\mu > 0$, implying that $P(t)$ has two positive roots and therefore $L(t)$ has a positive root $\geq 1$, by the symmetry $L(t) = L(1/t)$.
	Since $\lim_{t \rightarrow \infty} A(t) = + \infty$, we have that $A(t) \leq 0$ for $t \in [1/\l,\t]$ and $A(t) \geq 0$, otherwise. Thus 
	implying Part~\ref{part:theta-small} in the Lemma~\ref{lem:sign-change}. 
	
	If $\t \in (\pi/2,\tau)$, then $\mu < 0$ implying that both of the roots of $P(t)$ are non-positive. Since $\lim_{t \rightarrow \infty} A(t) = -\infty$
	it follows that $P(t) \leq 0$ for all $t>0$, which means that $L(t)\geq 0$, for all $t >0$.
	The case $\t = \pi/2$ is similar to this case, thus proving Part~\ref{part:theta-big} of Lemma~\ref{lem:sign-change}.\end{proof}

\vspace{4mm}

For the proof of Lemma~\ref{lem:small-log-bound}, we need the following lower bound for $A$, as defined at \eqref{equ:defA}.

\begin{obs} \label{obs:bound-on-A} For $\tau \in (0,\pi)$ and $\t$ satisfying $\pi/4 \leq |\t| \leq \pi$, let $A = A_{\theta,\tau}$ be as above.
	Then
	\[ -A(t) \geq \tau^2\left(1 + O(\tau^2)\right), \] for $t \in [1,3/2]$.
\end{obs}
\begin{proof} We write $A(t) = -4t H(t)$ and note that 
	\[ H(1) = (\sin \tau)^2 - 2(1-\cos \tau)\cos \t = \tau^2\left( 1-\cos\t + O(\tau^2)\right). \]
	Also, 
	\[ H'(t) = \tau^2 \left(\tau^{-2}(\sin \tau)^2 - 2t\tau^{-2}(1-\cos \tau)\cos \t\right)  = \tau^2\left( 1-t\cos \t + O(\tau^2) \right). \]
	So for $t \in [1,3/2]$ we have that 
	\begin{align*}H(t) &\geq \tau^2( 1 - \cos(\t) + O(\tau^2)) + (t-1)\tau^2(1 - t\cos(\t) + O(\tau^2) ) \\
		&\geq \tau^2 (t - (1 - t^2 + t)\cos(\t) + O(\tau^2) ) \\
		&\geq \frac{\tau^2}{4}(1 + O(\tau^2))\,.
	\end{align*}
	where in the final inequality we used the easy bounds $t - (1 - t^2 + t)\cos(\t) \geq t - (1 - t^2 + t)(1/\sqrt{2})$ by positivity of $1 - t^2 + t$ for $t \in [1,3/2]$ and also $t - (1 - t^2 + t)(1/\sqrt{2}) \geq 1/4$.  Using that $-A(t) = 4tH(t)$ finishes the proof.
\end{proof}

\vspace{4mm}

We can now prove Lemma~\ref{lem:small-log-bound}.

\begin{proof}[Proof of Lemma~\ref{lem:small-log-bound}]
	Set $L(t) = L_{\t,\tau}$. Since $\l^{-1} \leq 1-\eps$, we have 
	\[ \int_{1-\eps}^1 L(t)t^{-(s+1)}\, dt \geq \int_{1-\eps}^1 L(t)\, dt \]
	and using the symmetry $L(t) = L(1/t)$ and the fact $1 + \eps \leq (1-\eps)^{-1} \leq \l$, we have
	\[  \int_1^{\frac{1}{1-\eps}}  L(t) \,dt \geq \int_1^{1+\eps} L(t) \, dt. \]
	Set $t = 1+x$ and note that for $x \in [0,\eps] $ we have $|B(x)| \leq  16(1+\eps)^4$.
	So for $x \in [0,\eps]$, we can use Observation~\ref{obs:bound-on-A} along with the fact that $|\t| \geq \pi/4$ to obtain
	\[ -\alpha(1+x) = \frac{-A(1+x)}{B(1+x)} \geq \frac{\tau^2}{2^8}(1 + o_{\tau \rightarrow 0}(1)).
	\] So, using that $\log(1+x) \geq x/2$ for $x\in [0,1]$, we have 
	\[ \int_{0}^{\eps}\log\left( 1 -\alpha(1+x)\right) \, dx 
	\geq \frac{1}{2}\int_{0}^{\eps} \min\{-\alpha(1+x),1\}\, dx 
	\geq \frac{\eps \tau^2}{2^9}(1 + o_{\tau \rightarrow 0}(1)),\]
	as desired.
\end{proof}

%%% AUTHOR: optional acknowledgments here
\section*{Acknowledgments} %%  you may comment this out if no Ackno
We would like to thank Tyler Helmuth, Rob Morris and Will Perkins for useful comments on a draft and background information.

%%% AUTHOR:
%%% Bibliography goes here. Note that the arXiv cannot process bibtex
%%% or biber bibliographies.  Example of acceptable bibliograpy format:
\bibliographystyle{amsplain}

%% AUTHOR: You can generate such a bibliography from a .bib file by 
%% running pdflatex/bibtex/pdflatex/pdflatex and then pasting the .bbl file
%% between \begin{thebibliography} and \end{bibliography}

%%% AUTHOR: Include a short description of each author following the
%%% structure below. Use the same short tags used previously.  
%%% Use \imageat{} and \imagedot{} instead of "@" and "." in
%%% email addresses-this replaces the symbols with graphics to avoid 
%%% e-mail address harvesting from the .pdf file
\begin{dajauthors}
\begin{authorinfo}[marcus]
  Marcus Michelen\\
  University of Illinois at Chicago\\
  Chicago, Illinois, United States of America\\
  michelen.math\imageat{}gmail\imagedot{}com \\
  \url{https://marcusmichelen.org/}
\end{authorinfo}
\begin{authorinfo}[julian]
  Julian Sahasrabudhe\\
  University of Cambridge \\
  Cambridge, United Kingdom \\
  jdrs2\imageat{}cam\imagedot{}ac\imagedot{}uk \\
  \url{https://www.dpmms.cam.ac.uk/~jdrs2}
\end{authorinfo}

\end{dajauthors}

\end{document}